\documentclass[a4paper,12pt]{article}
\usepackage[centertags]{amsmath}
\usepackage{amsfonts}
\usepackage{amssymb}
\usepackage{amsthm}
\usepackage{amsmath}
\usepackage{dsfont}
\usepackage{graphicx} 
\usepackage{tikz, subfigure}
\usepackage{pstricks-add}
\usepackage{verbatim}

\usepackage[latin1]{inputenc}
\usepackage{tikz}
\definecolor {processblue}{cmyk}{0.96,0,0,0}
\usepackage{amsfonts,graphicx,amsmath,amssymb,hyperref,color}
\usepackage[english]{babel}
\usepackage{authblk}

\newtheorem{Theorem}{Theorem}[section]
\newtheorem{Proposition}{Proposition}[section]
\newtheorem{Lemma}{Lemma}[section]

\newtheorem{Example}{Example}[section]

\newtheorem{Remark}{Remark}[section]

\AtEndDocument{\bigskip{\footnotesize%
  \textsc{Department of Mathematics, Ningbo University, Ningbo, Zhejiang, People's Republic of China} \par
  \textit{E-mail address:} \texttt{kanjiangbunnik@yahoo.com, jiangkan@nbu.edu.cn} \par
}}
\AtEndDocument{\bigskip{\footnotesize%
  \textsc{Department of Mathematics, Ningbo University, Ningbo, Zhejiang, People's Republic of China} \par
  \textit{E-mail address:} \texttt{xilifeng@nbu.edu.cn} \par
}}

\usepackage{lipsum}

\makeatletter
\newcommand*{\rom}[1]{\expandafter\@slowromancap\romannumeral #1@}
\makeatother

\begin{document}
\title{xxxx}
\date{}
 \title{Multiplication on self-similar sets with overlaps}
\author{Li Tian, Jiangwen Gu, Qianqian, Ye, Lifeng Xi, and Kan Jiang\thanks{Corresponding author}}
\maketitle{}
\begin{abstract}
Let $A,B\subset\mathbb{R}$. Define $$A\cdot B=\{x\cdot y:x\in A, y\in B\}.$$
In this paper, we consider the following   class of self-similar sets with overlaps.
Let $K$ be the attractor of the IFS $\{f_1(x)=\lambda x, f_2(x)=\lambda x+c-\lambda,f_3(x)=\lambda x+1-\lambda\}$, where  $f_1(I)\cap f_2(I)\neq \emptyset, (f_1(I)\cup f_2(I))\cap f_3(I)=\emptyset,$
and  $I=[0,1]$ is the convex hull of $K$.
The main result of this paper is
$K\cdot K=[0,1]$ if and only if  $(1-\lambda)^2\leq c$. 
 Equivalently,   we give a  necessary  and sufficient condition  such that  for any $u\in[0,1]$, $u=x\cdot y$, where $x,y\in K$. 
\end{abstract}
\section{Introduction}
Given $A,B\subset\mathbb{R}$. Define $A*B=\{x*y:x\in A, y\in B\}$, where $*$ is  $+, -,\times$ or $\div$ (when $*=\div$, $y\neq 0$). The arithmetic sum  of two Cantor sets was studied by many scholars. There are many results concerning with this topic, see \cite {Yoccoz, Hochman,Dekking1,Dekking2,  Eroglu, SumKan}  and references therein.  It is an important problem in  homoclinic bifurcations \cite{Palis}.  The sum of two fractal sets  is similar to the projection of the product of these two sets through some angle \cite{Falconer}. Therefore,  one  can consider the sum of two fractal sets from the projection perspective \cite{Hochman, PS, Kan2019}. 
For the multiplication on  two  fractal sets, however, to the best of our knowledge, few papers  analyzed this  topic.   From the physical point of view, this problem  arises naturally in the study of the spectrum of the Labyrinth model \cite{Take}. 
In \cite{Tyson}, Athreya,  Reznick, and Tyson considered the multiplication and division on the middle-third Cantor sets.   They proved that $17/21\leq \mathcal{L}(C\cdot C)\leq 8/9$, where $\mathcal{L}$ denotes the Lebesgue measure and $C$ is the middle-third Cantor set. 
There are still many open questions. For instance,  if the middle-third Cantor set is replaced by the  overlapping self-similar sets \cite{Hutchinson}, then how can we obtain the sharp result, i.e. giving a necessary and sufficient condition such that  the multiplication of two overlapping self-similar sets is exactly some interval. This is  one of the main motivations of this paper. Another motivation of analyzing the multiplication  on self-similar sets is that we want to give a  new representation for any number in the unit interval, namely, given any $u\in[0,1]$, then how can we find $x,y$ in the same self-similar set such that  $u=x\cdot y$.

In this paper, we consider    the following  class of overlapping self-similar sets \cite{Hutchinson}. 
Let $K$ be the self-similar set of the  IFS $$\{f_1(x)=\lambda x, f_2(x)=\lambda x+c-\lambda,f_3(x)=\lambda x+1-\lambda, 0<\lambda<1\}.$$ 
We assume that $f_1(I)\cap f_2(I)\neq \emptyset, (f_1(I)\cup f_2(I))\cap f_3(I)=\emptyset,$
where $I=[0,1]$ is the convex hull of $K$. 
This class of self-similar set,  which  is indeed  a classical example  allowing  overlaps \cite{Hutchinson},  was investigated by many people. The celebrated conjecture posed by  Furstenberg 
states that  the self-similar set $$\Lambda=\dfrac{\Lambda}{3}\cup \dfrac{\Lambda+\gamma}{3}\cup\dfrac{\Lambda+2}{3} $$
has Hausdorff dimension $1$ for any irrational $\gamma. $  Hochman \cite{Hochman} proved  this conjecture is correct. 
Keyon \cite{Keyon},
 Rao and Wen \cite{Rao} studied the Hausdorff dimension of $\Lambda$ if $\gamma$ is rational. They proved that $\mathcal{H}^{1}(\Lambda)>0$ if and only if $\lambda=p/q\in \mathbb{Q}$ with $p\equiv q\not\equiv (0\equiv 3)$.  Ngai and Wang \cite{NW} came up with the finite type condtion, and gave an algorithm which can calculate the Haudorff dimension of $\Lambda$ when  $\Lambda$ is of finite type.  In \cite{DJKL, KarmaKan2, KarmaKan}, Dajani et al. analyzed the points in $\Lambda$ with multiple codings, and obtained that  when the overlaps are the  exact overlaps, then the set of points with  exactly $k$ codings has the same Hausdorff dimension as the univoque set. In \cite{Xi}, Guo et al. considered the bi-Lipschitz equivalence of overlapping self-similar sets when $\gamma$ differs. In \cite{JWX}, Jiang, Wang and Xi considered when the self-similar set $\Lambda$ is bi-Lipschitz equivalent to another self-similar set with  the strong separation condtion. All these results analyzed the overlapping self-similar sets from different aspects. 
 
 In this paper, we consider the multiplication on $K$. The assumptions on $K$ allow very compliciated overlaps.  We, however, have the following  result. 
\begin{Theorem}\label{Main}
Let $K$ be the self-similar set defined above. Then 
$$K\cdot K=[0,1] \mbox{ if and only if  } (1-\lambda)^2\leq c.$$
\end{Theorem}
\setcounter{Remark}{1}
\begin{Remark}\label{Remark}
The necessary condition is due to the following observation:
$$K\subset [0,c]\cup [1-\lambda, 1], \mbox{ which implies  }
K\cdot K\subset [0,c]\cup [(1-\lambda)^2,1].$$ 
\end{Remark}
This paper is arranged as follows. In section 2, we give a proof of Theorem \ref{Main}. In section 3, we give two examples. Finally, we give some remarks.
\section{Proof of Theorem \ref{Main}}
In this section, we first prove two useful lemmas. 
\subsection{Preliminaries}
Let $I=[0,1]$. For any $(i_1\cdots i_n)\in\{1,2,3\}^{n}$, we call $f_{i_1\cdots i_n}(I)$ a basic interval with length $\lambda^n$. Denote by $E_n$ the collection of  all the basic intervals with length $\lambda^n$. Let $J\in E_n$. Denote $\widetilde{J}=\cup_{i=1}^{3}I_{n+1,i}$, where $I_{n+1,i}\in E_{n+1}, I_{n+1,i}\subset J, i=1,2,3$.  Let $[A,B]\subset [0,1]$, where $A$ and $B$ are  the  left and right endpoints of  some basic intervals in $E_k$ for some $k\geq 1$, respectively. $A$ and $B$ may not in the same basic interval.  In the following lemma, we choose $A$ and $B$ in this way.  Let $F_k$ be the collection of all the basic intervals in $[A,B]$ with length  $\lambda^k, k\geq k_0$ for some $k_0\in\mathbb{N}^{+}$, i.e.  the union of all the elements of $F_k$ is denoted by $G_k=\cup_{i=1}^{t_k}I_{k,i}$, where $t_k\in \mathbb{N}^{+}$,  $I_{k,i}\in E_k$ and $I_{k,i}\subset [A,B]$. Clearly, by the definition of $G_n$, it follows that $G_{n+1}\subset G_n$ for any $n\geq k_0.$
\begin{Lemma}\label{key1}
Assume $F:\mathbb{R}^2\to \mathbb{R}$ is a continuous function.  
Suppose $A$ and $B$ are  the  left and right endpoints of  some basic intervals in $E_{k_0}$ for some $k_0\geq 1$, respectively. 
Then  $K\cap [A,B]=\cap_{n={k_0}}^{\infty}G_n$.
Moreover, if   for any $n\geq k_0$ and any basic intervals $I_1, I_2\subset G_n$, 
$$F(I_1, I_2)=F(\widetilde{I_1}, \widetilde{I_2}),$$
then $F(K\cap [A,B],K\cap [A,B] )=F(G_{k_0}, G_{k_0}).$ 
\end{Lemma}
\begin{proof}
Let $G_n=\cup_{i=1}^{t_n}I_{n,i}$ for some $t_n\in \mathbb{N}^{+}$, where $I_{n,i}\in E_n$ and $I_{n,i}\subset [A,B]$.    Then by the construction of $G_n$, i.e. $G_{n+1}\subset G_n$ for any $n\geq k_0$,  it follows that   $$K\cap [A,B]=\cap_{n=k_{0}}^{\infty}G_n.$$ By the continuity of $F$, we conclude that 
\begin{eqnarray}\label{identity}
F(K\cap  [A,B],K\cap  [A,B] )=\cap_{n=k_0}^{\infty}F(G_n, G_n).
\end{eqnarray}
By virtue of the relation $G_{n+1}=\widetilde{G_n}$ and the condition in the lemma, we have 
\begin{eqnarray*}
F(G_n, G_n)&=& \cup_{1\leq i,j\leq t_n}F(I_{n,i}, I_{n,j})\\&=
&\cup_{1\leq i,j\leq t_n}F(\widetilde{I_{n,i}},\widetilde{I_{n,j}})\\&=&
F(\cup_{1\leq i\leq t_n}\widetilde{I_{n,i}},\cup_{1\leq j\leq k_n}\widetilde{I_{n,j}})\\&=&F(G_{n+1}, G_{n+1}).
\end{eqnarray*}
Therefore,  $F(K\cap  [A,B],K\cap  [A,B] )=F(G_{k_0}, G_{k_0})$ follows immediately   from   identity (\ref{identity}) and $F(G_n, G_n)=F(G_{n+1}, G_{n+1})$ for any $n\geq k_0.$ 
\end{proof}
\begin{Lemma}\label{key2}
 Let $I_1=[a,a+t], I_2=[b,b+t]$ be two basic intervals. If $a\geq b\geq 1-c-\lambda$ and $c$ satisfies the following inequalities 
\begin{equation*}
\left\lbrace\begin{array}{cc}
                (1-\lambda)^2\leq c\\
                1-2c\leq \lambda\\
                \dfrac{1-c}{2}\leq \lambda\\
                \end{array}\right.
\end{equation*}
then $f(I_1, I_2)=f(\widetilde{I_1}, \widetilde{I_2}),$ where $f(x,y)=xy.$
\end{Lemma}
\begin{proof}
Since $I_1=[a,a+t], I_2=[b,b+t]$, it follows that 
$$\widetilde{I_1}=[a, a+ct]\cup [a+t-\lambda t, a+t], \widetilde{I_2}=[b, b+ct]\cup [b+t-\lambda t, b+t].$$
Clearly, 
$f(I_1, I_2)=[ab, (a+t)(b+t)]$. 
Now we calculate $f(\widetilde{I_1}, \widetilde{I_2})$.  Without loss of generality, we may assume $a\geq b$. After  simple calculation, 
$$f(\widetilde{I_1}, \widetilde{I_2})=J_1\cup J_2\cup J_2\cup J_4,$$
where 
\begin{eqnarray*}
J_1&=&[ab, (a+ct)(b+ct)]=[e_1,h_1]   \\
J_2&=& [b(a+dt), (a+t)(b+ct)]=[e_2, h_2]\\
J_3&=&[a(b+dt), (a+ct)(b+t)]=[e_3,h_3]\\
J_4&=& [ (a+dt)(b+dt), (a+t)(b+t)]=[e_4,h_4],
\end{eqnarray*}
 and $d=1-\lambda.$
Note that $f(\widetilde{I_1}, \widetilde{I_2})=f(I_1, I_2)=[ab, (a+t)(b+t)]$ if and only if 
\begin{equation*}
\left\lbrace\begin{array}{cc}
               h_1-e_2\geq 0\\
                h_2-e_3\geq0\\
                h_3-e_4 \geq0\\
                \end{array}\right.
\end{equation*}
Therefore, it suffices to prove the above inequalities. 

(I) If $a\geq b,$ then%
\begin{eqnarray*}
  h_1-e_2&=&(a+ct)(b+ct)-(a+dt)b \\
&=&t\left( c^{2}t+ac+bc-bd\right)  \\
&\geq &t\left( c^{2}t+bc+bc-bd\right)  \\
&=&t\left( c^{2}t+b(2c-d)\right) \geq 0.
\end{eqnarray*}
Therefore, if  
\begin{equation*}
2c-d\geq 0 (\Leftrightarrow 2c\geq 1-\lambda)
\end{equation*}
then 
\begin{equation*}
(a+ct)(b+ct)-(a+dt)b\geq 0.
\end{equation*}
(II) We  need to show $$h_2-e_3=(a+t)(b+ct)-a(b+dt)=t\left( b+ac-ad+ct\right)
=ct^{2}+t(b+ac-ad)\geq 0.$$
In fact, the following inequality is sufficient,
\begin{equation*}
\sup_{a\in \lbrack b,1]}a(d-c)\leq b
\end{equation*}%
i.e., 
\begin{equation*}
1-\lambda-c=(d-c)\leq b, 
\end{equation*}
which is the assumption in lemma. 

(III) If $a\geq b,$ then%
\begin{eqnarray*}
h_3-e_4&=&(b+t)(a+ct)-(a+dt)(b+dt) \\
&=&t\left( a-d^{2}t-ad+bc-bd+ct\right). 
\end{eqnarray*}
$\allowbreak $It suffices to prove that $a-d^{2}t-ad+bc-bd+ct\geq 0$ if $%
a\geq b.$ When $a\geq b,$ we obtain that $a-ad\geq b-bd$ as $%
a-ad-(b-bd)=\left( 1-d\right) \left( a-b\right) \geq 0.$ As such
\begin{eqnarray*}
&&a-d^{2}t-ad+bc-bd+ct \\
&\geq &t\left( c-d^{2}\right) +b-bd+bc-bd \\
&\geq &t\left( c-d^{2}\right) +b(1-2d+c).
\end{eqnarray*}%
If  $c-d^{2}\geq 0$ and  \ 
\begin{equation*}
1-2d+c\geq 0(\Leftrightarrow c\geq 1-2\lambda)
\end{equation*}%
then 
\begin{equation*}
a-d^{2}t-ad+bc-bd+ct\geq 0
\end{equation*}%
Under  the condition 
\begin{equation*}
a\geq b\geq d-c=1-\lambda-c,
\end{equation*}
if  $c$ and $d$ satisfy the following inequalities 
\begin{equation*}
1-2d+c\geq 0,\text{ }2c\geq d\geq c\text{ and }c\geq d^{2},
\end{equation*}
then 
 $f(I_1, I_2)=f(\widetilde{I_1}, \widetilde{I_2}).$
\end{proof}
\subsection{Proofs of some lemmas}
We  first give an outline of the proof of Theorem \ref{Main}.  First, 
by the conditions for  $c$ and $\lambda$, see Lemma \ref{nec} and Remark \ref{Remark},  we have  $$\lambda\leq c\leq 2\lambda, (1-\lambda)^2\leq c <1-\lambda. $$ 
Therefore,   if $K\cdot K=[0,1]$,  then $(\lambda,c)$ should be in the  purple region (the first picture of Figure 1). 
Conversely, we shall prove that for any $(\lambda,c)$ in the purple region, $K\cdot K=[0,1]$. 
We  partition the purple region into five subregions, see  the last picture of Figure 1. 
More precisely,  in Lemma \ref{useful}, we prove that for the brown region in the last picture, $K\cdot K=[0,1]$.
In Lemma \ref{lem2}, we prove that for the gray region (the second picture), $K\cdot K=[0,1]$. 
In Lemma \ref{lem3}, we show that if $(\lambda,c)$ in the orange region (the third picture), then $K\cdot K=[0,1].$ In Lemma \ref{lem4}, when $(\lambda,c)$ in the blue  region (the fourth picture), we prove $K\cdot K=[0,1].$
Note that the  union of the regions generated  from Lemma \ref{lem2},  Lemma \ref{lem3} and Lemma \ref{lem4} is precisely the purple region in the first picture. 

Before, we prove Lemmas \ref{useful}, \ref{lem2},   \ref{lem3} and \ref{lem4}. We give  the following lemmas which are useful  to our analysis. 
\begin{Lemma}\label{nec}
Let $K$ be the self-similar set of the following  IFS $$\{f_1(x)=\lambda x, f_2(x)=\lambda x+c-\lambda,f_3(x)=\lambda x+1-\lambda, 0<\lambda<1\}.$$ 
If  $f_1(I)\cap f_2(I)\neq \emptyset, (f_1(I)\cup f_2(I))\cap f_3(I)=\emptyset,$
then $\lambda\leq c\leq 2\lambda,  0<\lambda<1-c.$ If $K\cdot K=[0,1]$, then $\lambda\geq 2-\sqrt{3}$. 
\end{Lemma}
\begin{proof}
The first statement is trivial. We only prove the second one. 
Since  $\lambda\leq c\leq 2 \lambda$, it follows that 
 $K\cdot K\subset [0,2\lambda]\cup [(1-\lambda)^2,1].$ If $0<\lambda< 2-\sqrt{3},$ then
$2\lambda<(1-\lambda)^2$, which contradicts with $K\cdot K=[0,1].$ 
\end{proof}

\begin{Lemma}\label{key3}
If $(\lambda, c)$ satisfies the following conditions 
\begin{equation*}
\left\lbrace\begin{array}{cc}
                (1-\lambda)^2\leq c<1-\lambda\\
                \lambda\leq c\leq 2 \lambda\\
                \end{array}\right.
\end{equation*}
then 
\begin{equation*}
\left\lbrace\begin{array}{cc}
                1-2c\leq \lambda\\
                \dfrac{1-c}{2}\leq \lambda.\\
                \end{array}\right.
\end{equation*}
\end{Lemma}
\begin{proof}
The proof is due to the first picture of Figure 1. 
\end{proof}
In terms of this Lemma \ref{key3}, in Lemma \ref{key2}, the condition    
\begin{equation*}
\left\lbrace\begin{array}{cc}
 (1-\lambda)^2\leq c\\
                1-2c\leq \lambda\\
                \dfrac{1-c}{2}\leq \lambda\\
                \end{array}\right.
\end{equation*} can be replaced by 
\begin{equation*}
\left\lbrace\begin{array}{cc}
                (1-\lambda)^2\leq c<1-\lambda\\
                \lambda\leq c\leq 2 \lambda.\\
                \end{array}\right.
\end{equation*}

\begin{figure}
  \centering
  \includegraphics[width=146pt]{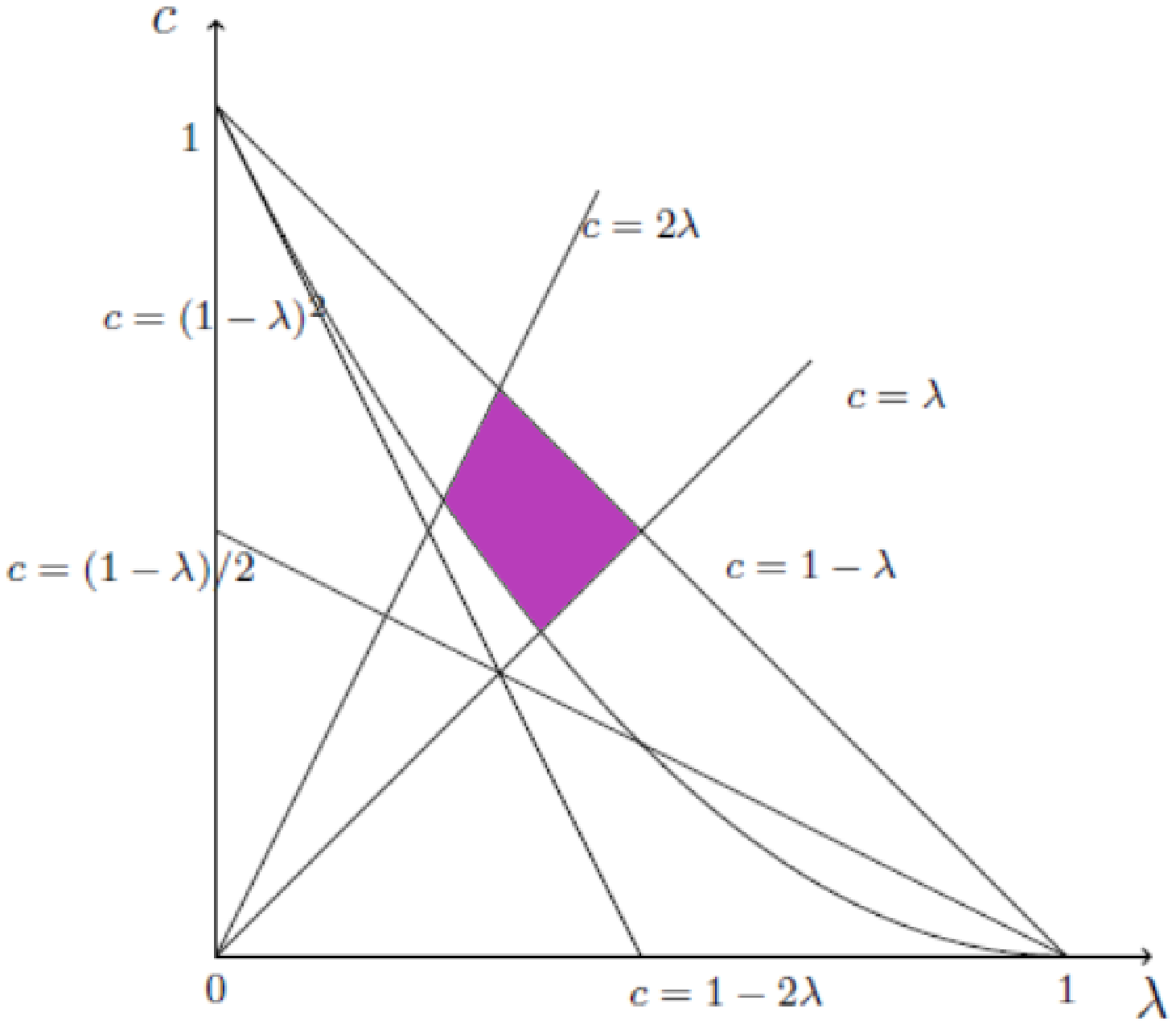}
   \includegraphics[width=146pt]{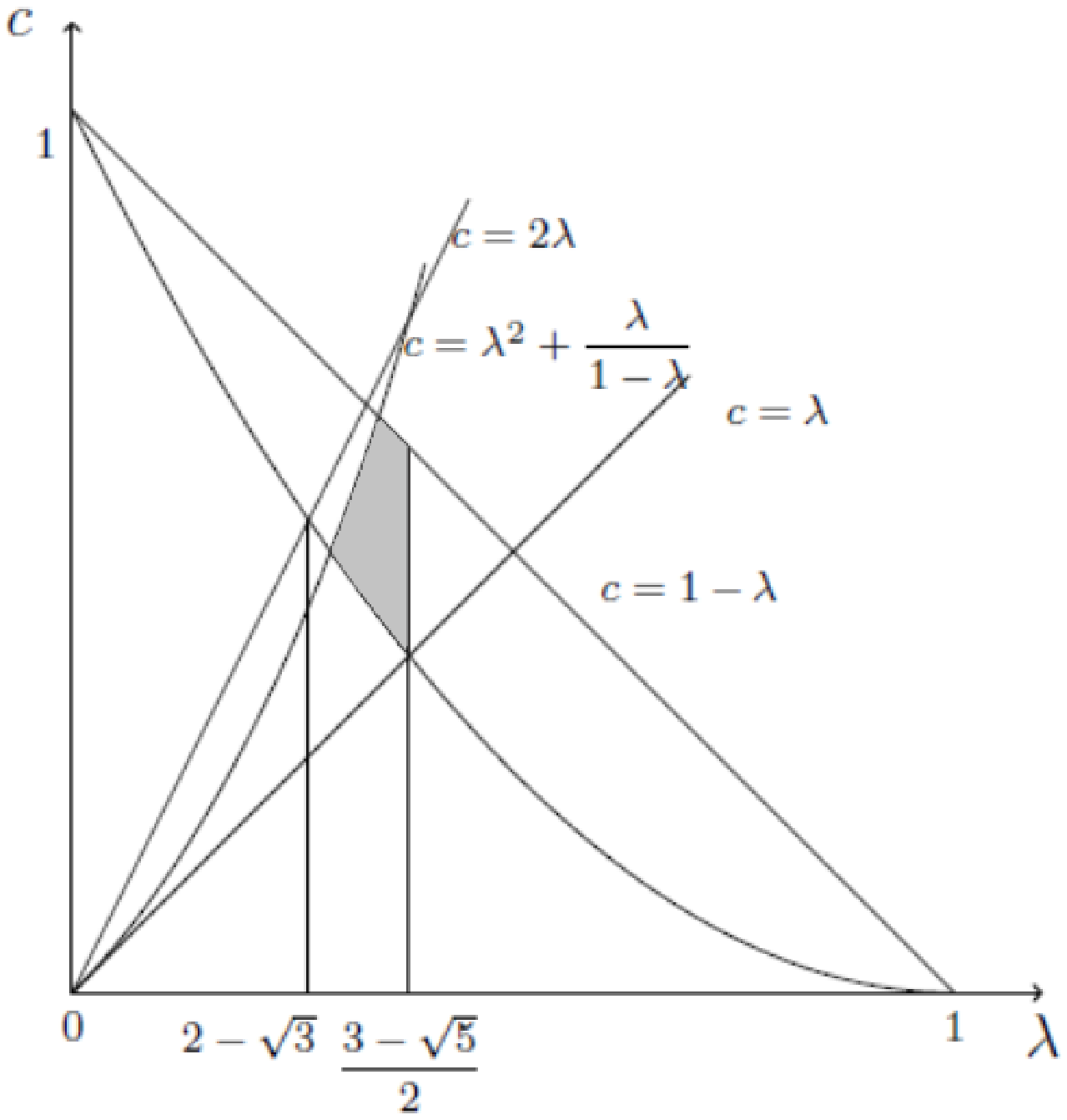}
    \includegraphics[width=146pt]{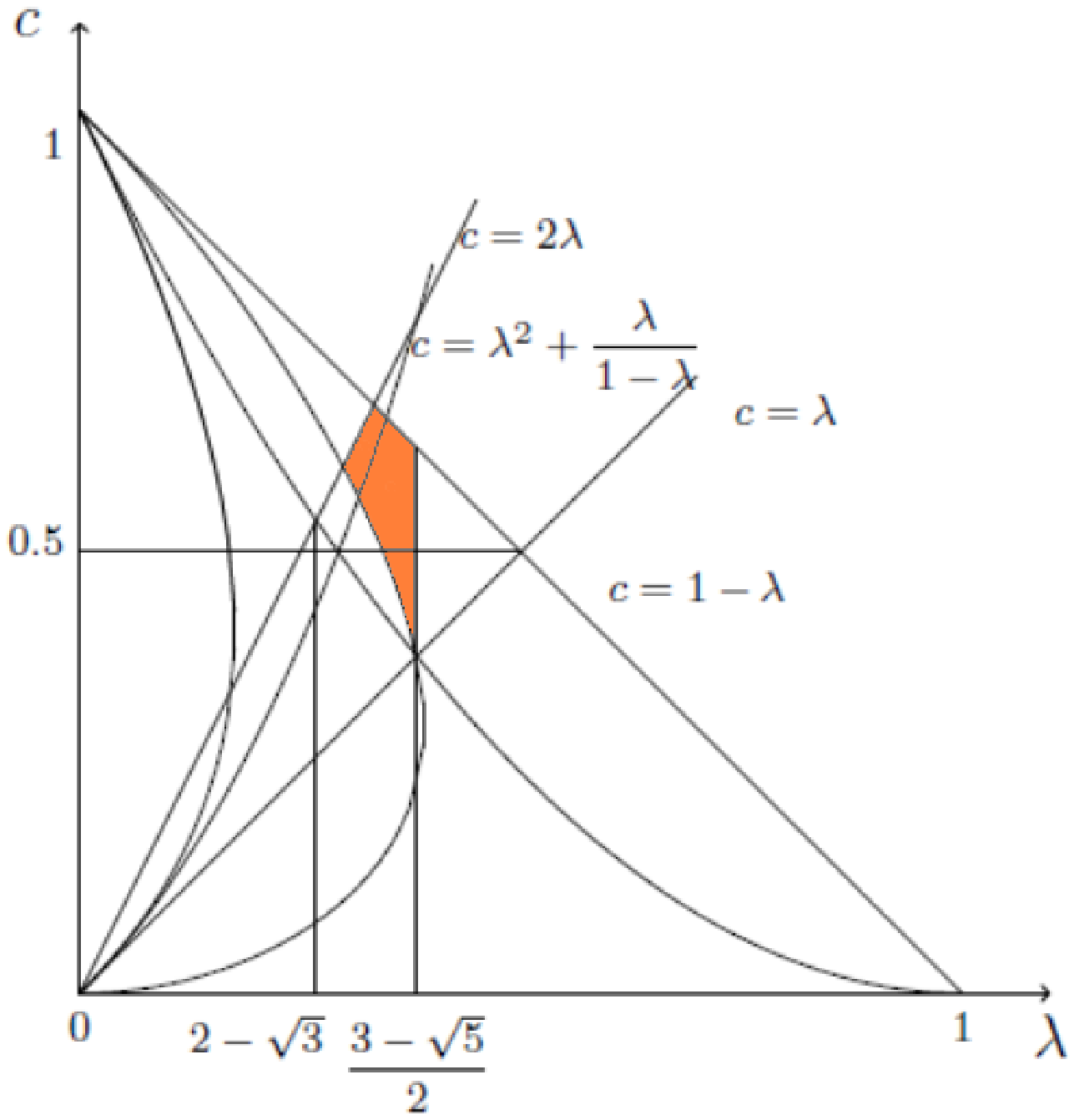}
       \includegraphics[width=146pt]{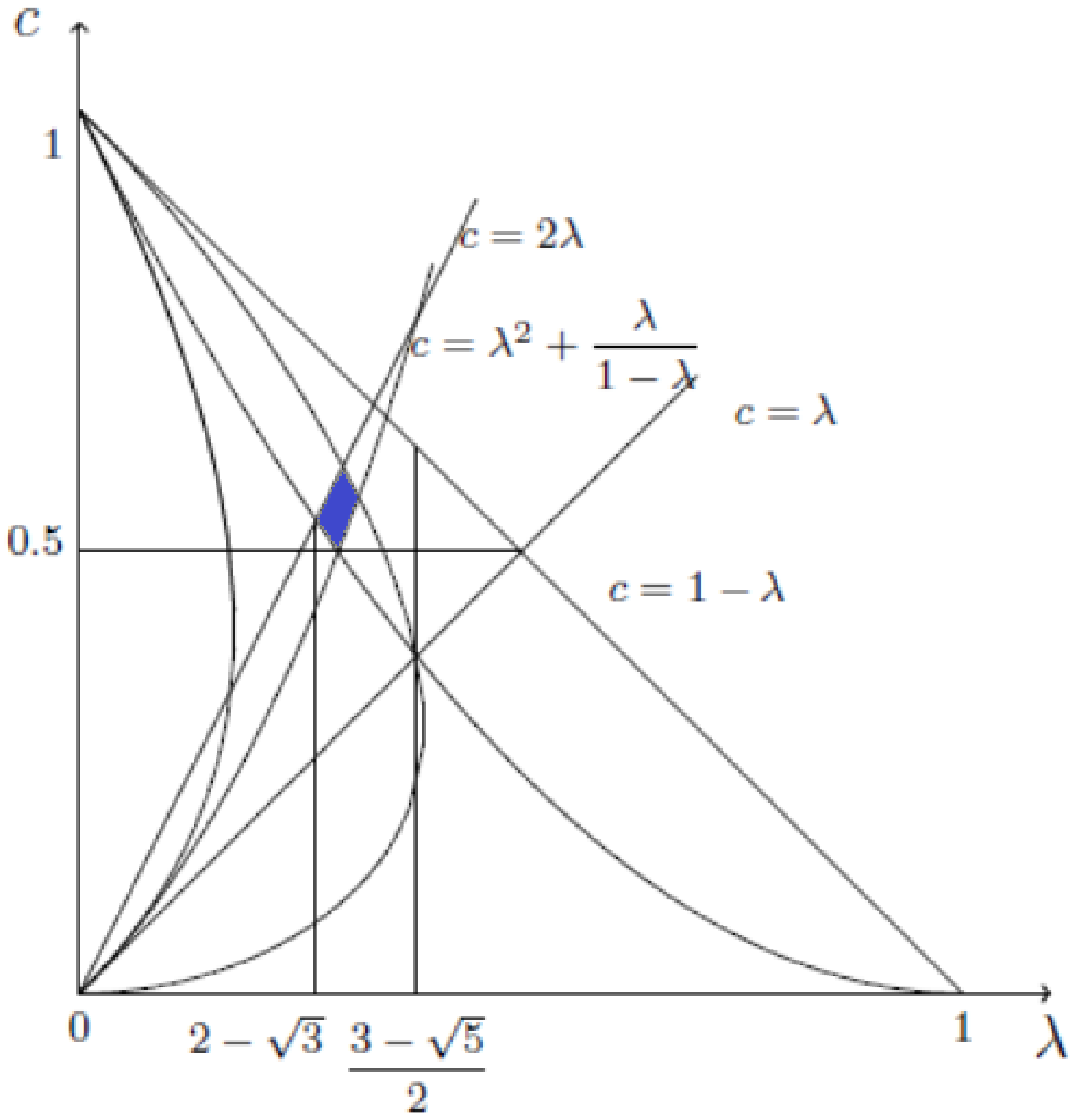}
      \includegraphics[width=146pt]{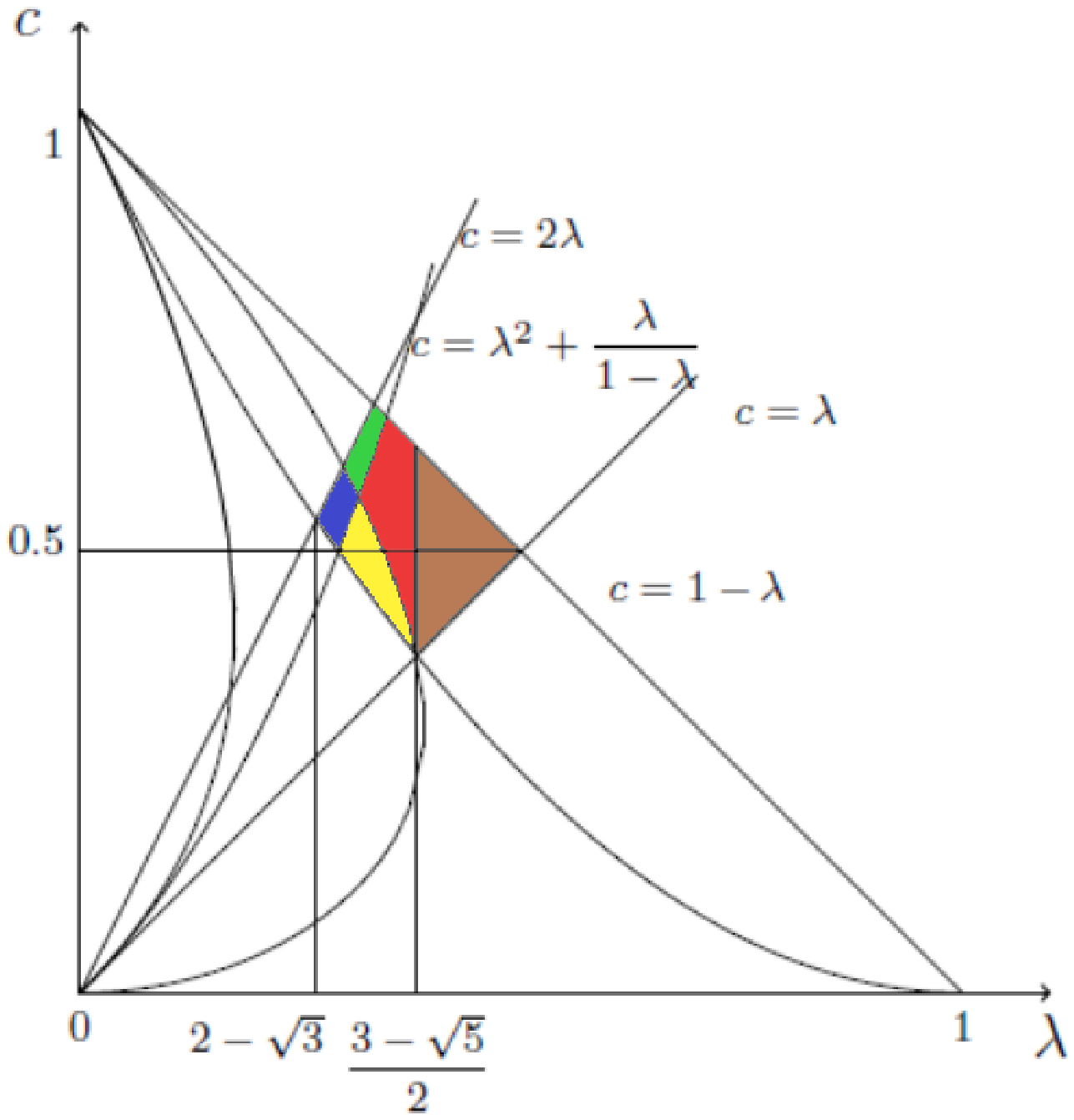}
   \caption{ Lemma \ref{useful} (brown region in the last picture), Lemma \ref{lem2} (gray region),  Lemma \ref{lem3} (orange region), Lemma \ref{lem4} (blue region)}
\end{figure}
\begin{Lemma}\label{useful}
If $(\lambda, c)$ satisfies the following conditions 
\begin{equation*}
\left\lbrace\begin{array}{cc}
                (1-\lambda)^2\leq c<1-\lambda\\
                \lambda\leq c\leq 2 \lambda\\
                \dfrac{3-\sqrt{5}}{2}\leq \lambda<1,
                \end{array}\right.
\end{equation*}
 then $K\cdot K=[0,1].$
\end{Lemma}
\begin{proof}
Since $1-\lambda\geq 1-c-\lambda$, 
in terms of Lemmas \ref{key1} and \ref{key2}, we let $[A,B]=[1-\lambda, 1]$. It is easy to check that 
$$f(K\cap [A,B], K\cap [A,B] )=[(1-\lambda)^2,1].$$ Therefore, 
$$ [0,1]\supset K\cdot K\supset \cup_{n=0}^{\infty}\lambda^n[(1-\lambda)^2,1]\cup\{0\}=[0,1].$$
\end{proof}

\begin{Lemma}\label{lem2}
For any $0<\lambda<1$,  the following inequality holds
 $$c(1-\lambda+\lambda c)>(c-\lambda^2)(1-\lambda^2).$$
Suppose $c$ and $\lambda$ satisfy the following inequalities,
\begin{equation*}
\left\lbrace\begin{array}{cc}
2-\sqrt{3}\leq \lambda< \dfrac{3-\sqrt{5}}{2}\\ 
                    (1-\lambda)^2\leq c<1-\lambda\\
                \lambda\leq c\leq 2 \lambda\\
                c\leq \lambda^2+\dfrac{\lambda}{1-\lambda}
                \end{array}\right.
\end{equation*}
 then  $K\cdot K=[0,1].$
\end{Lemma}
\begin{proof}
First, we prove that if $0<\lambda<1$, then $c(1-\lambda+\lambda c)>(c-\lambda^2)(1-\lambda^2)$. 

It is easy to check that  
$$x_1=\dfrac{-(\lambda-1)-\sqrt{(\lambda-1)^2-4(\lambda-\lambda^3)}}{2}, x_2=\dfrac{-(\lambda-1)+\sqrt{(\lambda-1)^2-4(\lambda-\lambda^3)}}{2}$$
are the roots of $c^2+(\lambda-1)c+\lambda-\lambda^3=0$. Since $0<\lambda<1$, it follows that 
$$(\lambda-1)^2-4(\lambda-\lambda^3)<0.$$ In other words, $x_1$ and $x_2$ are complex numbers rather than the reals.  Therefore, $$c^2+(\lambda-1)c+\lambda-\lambda^3\geq0.$$ 
With a similar discussion of  Lemma \ref{useful}, it follows that 
$$ K\cdot K\supset \cup_{n=0}^{\infty}\lambda^n([(1-\lambda)^2, 1]).$$
  By the assumptions $$2-\sqrt{3}\leq \lambda< \dfrac{3-\sqrt{5}}{2}, (1-\lambda)^2\leq c,$$ we conclude that 
$$2c\geq 2(1-\lambda)^2\geq 1-\lambda+\lambda^2.$$
In other words, $c-\lambda^2\geq 1-c-\lambda.$ Therefore, we may make use of Lemma \ref{key2} for $\widetilde{K}=K\cap [c-\lambda^2, 1]=([c-\lambda^2,c]\cup[1-\lambda,1-\lambda+c\lambda ]\cup[1-\lambda^2,1])\cap K$. 
Simple calculation yields that 
\begin{eqnarray*}
K\cdot K\supset f(\widetilde{K},\widetilde{K})&=&[(c-\lambda^2)^2, c^2]\cup [(c-\lambda^2)(1-\lambda),c(1-\lambda+c\lambda)] \\
&\cup&[ (c-\lambda^2)(1-\lambda^2), c]\cup [(1-\lambda)^2,(1-\lambda+c\lambda)^2]\\
&\cup&   [(1-\lambda^2)(1-\lambda), 1-\lambda+c\lambda]\cup [(1-\lambda^2)^2,1].
\end{eqnarray*}
  By the condition $(1-\lambda)^2\leq c$ and 
the consequence $$ K\cdot K\supset \cup_{n=0}^{\infty}\lambda^n([(1-\lambda)^2, 1]),$$
 we obtain that $K\cdot K\supset \cup_{n=1}^{\infty}\lambda^n[(c-\lambda^2)(1-\lambda),1].$
Since  $c\leq \lambda^2+\dfrac{\lambda}{1-\lambda}$, it follows that $\lambda\geq(c-\lambda^2)(1-\lambda)$. Therefore, 
$$[0,1]\supset K\cdot K\supset \cup_{n=0}^{\infty}\lambda^n[(c-\lambda^2)(1-\lambda),1]\cup\{0\}=[0,1],$$ as required. 
\end{proof}

\begin{Lemma}\label{lem3}
Suppose $c$ and $\lambda$ satisfy the following inequatlies,
\begin{equation*}
\left\lbrace\begin{array}{cc}
      2-\sqrt{3}\leq \lambda\leq \dfrac{3-\sqrt{5}}{2}\\
              (1-\lambda)^2\leq c<1-\lambda\\
                \lambda\leq c\leq 2 \lambda\\
               (c- \lambda^2)(1-\lambda) \leq   c^2\\
                \end{array}\right.
\end{equation*}
Then $K\cdot K=[0,1].$
\end{Lemma}
\begin{proof}
If $(c- \lambda^2)(1-\lambda) \leq   c^2$, then with a similar discussion as the proof of Lemma \ref{lem2}, we have the following inclusion 
$$[0,1]\supset K\cdot K\supset \cup_{n=0}^{\infty}\lambda^n[(c-\lambda^2)^2,1]\cup\{0\}=[0,1].$$
Here we need to assume $\lambda\geq (c-\lambda^2)^2$, that is, $c\leq \sqrt{\lambda}+\lambda^2. $ Since $c\leq 2\lambda$, it suffices to prove $2\lambda\leq  \sqrt{\lambda}+\lambda^2 $, which is a direct consequence of  $ 2-\sqrt{3}\leq \lambda\leq \dfrac{3-\sqrt{5}}{2}$. 
\end{proof}
\begin{Lemma}\label{key3}
Suppose $c$ and $\lambda$ satisfy the following inequatlies,
\begin{equation*}
\left\lbrace\begin{array}{cc}
2-\sqrt{3}\leq \lambda< \dfrac{3-\sqrt{5}}{2}\\
                (1-\lambda)^2\leq c \leq 2\lambda\\
               (c- \lambda^2)(1-\lambda) \geq   c^2\\
               c\geq \lambda^2+\dfrac{\lambda}{1-\lambda}
                \end{array}\right.
\end{equation*}
Then $c\geq \dfrac{1}{2}$, i.e. $c-\lambda\geq 1-c-\lambda.$
\end{Lemma}
\begin{proof}
The proof is due to the help of computer, see the fourth picture in Figure 1. We plot the blue region  which satisfies the conditions in lemma, and find that $c\geq \dfrac{1}{2}.$
\end{proof}
\begin{Lemma}\label{lem4}
Suppose $c$ and $\lambda$ satisfy the following inequalities,
\begin{equation*}
\left\lbrace\begin{array}{cc}
2-\sqrt{3}\leq \lambda< \dfrac{3-\sqrt{5}}{2}\\
                (1-\lambda)^2\leq c \leq 2\lambda\\
               (c- \lambda^2)(1-\lambda) \geq   c^2\\
               c\geq \lambda^2+\dfrac{\lambda}{1-\lambda}
                \end{array}\right.
\end{equation*}
Then $ K\cdot K=[0,1].$
\end{Lemma}
\begin{proof}
By Lemma \ref{key3},   it follows that $c-\lambda\geq 1-c-\lambda$. Therefore, we can utilize Lemmas \ref{key2} and \ref{key1} by taking $\widetilde{K}=([c-\lambda,c]\cup [1-\lambda,1])\cap K$.
Therefore, 
\begin{eqnarray*}
f(\widetilde{K},\widetilde{K})&=& f([c-\lambda,c]\cup [1-\lambda,1]),[c-\lambda,c]\cup [1-\lambda,1]))\\&=
&[(c-\lambda)^2, c^2]\cup [(c-\lambda)(1-\lambda),c]\cup [(1-\lambda)^2,1]
\end{eqnarray*}
The equation $c^2=(c-\lambda)(1-\lambda)$ has two roots, i.e.
$$c_1=\dfrac{-(\lambda-1)-\sqrt{(\lambda-1)^2-4(\lambda-\lambda^2)}}{2},c_2=\dfrac{-(\lambda-1)+\sqrt{(\lambda-1)^2-4(\lambda-\lambda^2)}}{2}.$$
Since $2-\sqrt{3}\leq \lambda< \dfrac{3-\sqrt{5}}{2}$, it follows that  $$(\lambda-1)^2-4(\lambda-\lambda^2)<0.$$
Therefore, $c^2>(c-\lambda)(1-\lambda)$. Subsequently, 
$$f(\widetilde{K},\widetilde{K})=[(c-\lambda)^2,1]$$
 Next, we prove that $\lambda\geq(c-\lambda)^2$ which is equivalent to $\sqrt{\lambda}+\lambda\geq c$. This is trivial as $\sqrt{\lambda}+\lambda\geq 2\lambda\geq c.$ 
 Therefore, 
$$[0,1]\supset K\cdot K\supset \cup_{n=0}^{\infty}\lambda^n[(c-\lambda)^2,1]\cup\{0\}=[0,1].$$
\end{proof}
\begin{proof}[\textbf{Proof of Theorem \ref{Main}}]
By  Remark \ref{Remark}, we only need to prove
the sufficiency of Theorem \ref{Main}.  If $c\geq  (1-\lambda)^2$ and $\dfrac{3-\sqrt{5}}{2}\leq \lambda\leq 1$, then by Lemma \ref{useful}, $K\cdot K=[0,1]$. If $2-\sqrt{3}\leq \lambda< \dfrac{3-\sqrt{5}}{2}$ and $c\geq  (1-\lambda)^2$, then by  Lemmas \ref{lem2}, \ref{lem3} and \ref{lem4}, $K\cdot K=[0,1]$. Note that  the  union of the associated regions of $(\lambda,c)$  satisfying the conditions in Lemmas \ref{useful}, \ref{lem2}, \ref{lem3} and \ref{lem4}   is the  brown, gray,  orange,  blue regions in the  pictures of Figure 1, which is exactly the purple region of  the first picture  of Figure 1.  
\end{proof}
\section{Examples}
\begin{Example}
Let  $K$ be the attractor of the following IFS, 
$$\left\{f_1(x)=\dfrac{x}{3}, f_2(x)=\dfrac{x}{3}+c-\dfrac{1}{3}, f_3(x)=\dfrac{x+2}{3}\right\},\dfrac{1}{3}\leq c<\dfrac{2}{3}.$$
If $c \in\left[\dfrac{4}{9}, \dfrac{2}{3}\right)$, then $K\cdot K=[0,1].$ Moreover, $c=\dfrac{4}{9}$ is sharp, i.e. for any $\dfrac{1}{3}\leq c<\dfrac{4}{9}$, $$K\cdot K\subsetneq[0,1].$$
\end{Example}
In this example $2-\sqrt{3}\leq \lambda=1/3<1$. Hence for  any $c\geq (1-\lambda)^2=\dfrac{4}{9}$, i.e. $c\in \left[\dfrac{4}{9},\dfrac{2}{3}\right)$, then $K\cdot K=[0,1].$ Moreover, $c=\dfrac{4}{9}$ is sharp. 
\begin{Example}
Let $K$ be the self-similar set of the following  IFS, $$\{f_1(x)=\lambda x, f_2(x)=\lambda x+\lambda-\lambda^n,f_3(x)=\lambda x+1-\lambda\},$$ where $0<\lambda<\beta, n\geq 2$, and $\beta\in(0,1)$ is the smallest real root of $x^n-3x+1=0$.
Then 
$K\cdot K=[0,1]$ if and only if $\alpha\leq \lambda<\beta$, where $\alpha\in(0,1)$ is the smallest  real root of $x^n+x^2-4x+1=0$.
\end{Example}
First, we prove the  following lemma.  
\setcounter{Proposition}{2}
\begin{Proposition}\label{compare}
Let $\alpha$ and $\beta$ be the smallest real roots of $x^n+x^2-4x+1=0$ and $x^n-3x+1=0$, respectively.  Then $\alpha<\beta$. 
\end{Proposition}
\begin{proof}
By  the Rouch\'{e} theorem and the intermediate value theorem, it follows that $\alpha$ and $\beta$ are the unique real roots of $x^n+x^2-4x+1=0$ and $x^n-3x+1=0$ in $(0,1)$, respectively.
Now, we prove  $\alpha<\beta$. 
If $\beta\leq \alpha$. Let $H(x)=x^n-3x+1$. Then 
$$H(\alpha)=\alpha^n-3\alpha+1=\alpha^n+\alpha^2-4\alpha+1+\alpha-\alpha^2=\alpha-\alpha^2\geq0.$$
$H(1)<0$. Therefore, by the intermediate value theorem, we conclude that there is another root of $H(x)$ in $(\alpha,1)$, which contradicts to the uniqueness of the root in $(0,1)$. 
\end{proof}
By Theorem \ref{Main}, $K\cdot K=[0,1]$ if and only if $c\geq (1-\lambda)^2$, where $c=2\lambda-\lambda^n$.  Therefore, $\lambda^n+\lambda^2-4\lambda+1\leq 0$. 
 The condition $\lambda^n-3\lambda+1>0$ is equivalent to $c<1-\lambda.$
 Therefore, $K\cdot K=[0,1]$ if and only if $\alpha\leq \lambda<\beta$. 
\section{Final remarks}
In this paper, we only consider the multiplication on the self-similar sets. It is natural to consider the division on the overlapping self-similar sets. Moreover,  we can prove the following result. 
\begin{Theorem}
Let $K$ be the attractor defined in Theorem \ref{Main}. 
Given $u\in [0,1]$, then there exist  some $\lambda, c$ and some $x_1,x_2,x_3,x_4, x_5, x_6\in K$ such that 
$$u=x_1\cdot x_2=x_3^2+x_4^2=\dfrac{x_5}{x_6}.$$
\end{Theorem}
  We will publish these results elsewhere.
  \section*{Acknowledgements}
The work is supported by National Natural Science Foundation of China (Nos.11771226, 11701302,
11371329, 11471124, 11671147). The work is also supported by K.C. Wong Magna Fund in Ningbo University.


\begin{thebibliography}{10}

\bibitem{Yoccoz}
Carlos Gustavo~T. de~A.~Moreira and Jean-Christophe Yoccoz.
\newblock Stable intersections of regular {C}antor sets with large {H}ausdorff
  dimensions.
\newblock {\em Ann. of Math. (2)}, 154(1):45--96, 2001.

\bibitem{DJKL}
Karma Dajani, Kan Jiang,  Derong Kong, and Wenxia Li.
\newblock Multiple expansions of real numbers with digits set $\{0,1,q\}$.
\newblock {\em arXiv:1508.06138}, 2015.

\bibitem{KarmaKan2}
Karma Dajani and Kan Jiang,  Derong Kong, and Wenxia Li. 
\newblock Multiple codings for self-similar sets with overlaps.
\newblock {\em arXiv:1603.09304}, 2016.

\bibitem{Dekking2}
F.~Michel Dekking and Bram Kuijvenhoven.
\newblock Differences of random {C}antor sets and lower spectral radii.
\newblock {\em J. Eur. Math. Soc. (JEMS)}, 13(3):733--760, 2011.

\bibitem{Dekking1}
Michel Dekking and K{\'a}roly Simon.
\newblock On the size of the algebraic difference of two random {C}antor sets.
\newblock {\em Random Structures Algorithms}, 32(2):205--222, 2008.


\bibitem{Eroglu}
Kemal~Ilgar Ero{\u{g}}lu.
\newblock On the arithmetic sums of {C}antor sets.
\newblock {\em Nonlinearity}, 20(5):1145--1161, 2007.

\bibitem{Falconer}
K.~J. Falconer.
\newblock {\em The geometry of fractal sets}, volume~85 of {\em Cambridge
  Tracts in Mathematics}.
\newblock Cambridge University Press, Cambridge, 1986.


\bibitem{Xi}
Qiuli Guo, Hao Li, Qin Wang, and Lifeng Xi.
\newblock Lipschitz equivalence of a class of self-similar sets with complete
  overlaps.
\newblock {\em Ann. Acad. Sci. Fenn. Math.}, 37(1):229--243, 2012.

\bibitem{Hochman}
Michael Hochman.
\newblock On self-similar sets with overlaps and inverse theorems for entropy.
\newblock {\em Ann. of Math. (2)}, 180(2):773--822, 2014.

\bibitem{Hutchinson}
John~E. Hutchinson.
\newblock Fractals and self-similarity.
\newblock {\em Indiana Univ. Math. J.}, 30(5):713--747, 1981.

\bibitem{SumKan}
Kan Jiang. \newblock Hausdorff dimension of the arithmetic sum of self-similar sets. 
\newblock{\em Indag. Math. (N.S.)},27(3):684-701, 2016.


\bibitem{KarmaKan}
 K. Jiang and K. Dajani.
\newblock Subshifts of finite type and self-similar sets.
\newblock {\em Nonlinearity}, 30(2):659--686, 2017.


\bibitem{JWX}
Kan Jiang, Songjing Wang, and Li-Feng Xi.
\newblock Lipschitz equivalence of self-similar sets with exact overlaps.
\newblock {\em Accepted by Ann. Acad. Sci. Fenn. Math.}, 2017.

\bibitem{Kan2019}
Kan Jiang.
\newblock Projections of cartesian products of the self-similar sets without
  the irrationality assumption.
\newblock {\em arXiv:1806.01080}, 2018.

\bibitem{Keyon}
R.~Keyon.
\newblock Projecting the one-dimensional sierpinski gasket.
\newblock {\em Israel J. Math.}, 97(221-238), 1997.

\bibitem{NW}
Sze-Man Ngai and Yang Wang.
\newblock Hausdorff dimension of self-similar sets with overlaps.
\newblock {\em J. London Math. Soc. (2)}, 63(3):655--672, 2001.

\bibitem{Palis}
Jacob Palis and Floris Takens.
\newblock {\em Hyperbolicity and sensitive chaotic dynamics at homoclinic
  bifurcations}, volume~35 of {\em Cambridge Studies in Advanced Mathematics}.
\newblock Cambridge University Press, Cambridge, 1993.
\newblock Fractal dimensions and infinitely many attractors.

\bibitem{PS}
Yuval Peres and Pablo Shmerkin.
\newblock Resonance between {C}antor sets.
\newblock {\em Ergodic Theory Dynam. Systems}, 29(1):201--221, 2009.



\bibitem{Rao}
Hui Rao and Zhi-Ying Wen.
\newblock A class of self-similar fractals with overlap structure.
\newblock {\em Adv. in Appl. Math.}, 20(1):50--72, 1998.

\bibitem{Tyson}
Jayadev S.Athreya, Bruce Reznick, and Jeremy T.Tyson.
\newblock Cantor set arithmetic.
\newblock {\em arXiv:1711.08791}, 2017.

\bibitem{Take}
Yuki Takahashi.
\newblock Quantum and spectral properties of the {L}abyrinth model.
\newblock {\em J.Math.Phys.}, 57, 2016.

\end{thebibliography}

\end{document}